\documentclass[reqno]{amsart}
\usepackage{amssymb,amsthm,amscd}
\usepackage[pdftex]{graphics}
\input{diagrams}

\newcommand{\fp}{{\mathfrak p}}
\newcommand{\fP}{{\mathfrak P}}

\newcommand{\Cl}{{\operatorname{Cl}}}
\newcommand{\Gal}{{\operatorname{Gal}}}

\newcommand{\Z}{{\mathbb Z}}
\newcommand{\Q}{{\mathbb Q}}
\newcommand{\bQ}{\overline{\mathbb Q}}

\newcommand{\cO}{{\mathcal O}}
\newcommand{\eps}{\varepsilon}

\newcommand{\lra}{\longrightarrow}
\newcommand{\la}{\langle}
\newcommand{\ra}{\rangle}

\newcommand{\bK}{\overline{K}}
\newcommand{\bL}{\overline{L}}

\newtheorem{thm}{Theorem}[section]
\newtheorem{prop}[thm]{Proposition}
\newtheorem{lem}[thm]{Lemma}

\numberwithin{equation}{section}

\title{Harbingers of Artin's Reciprocity Law. \\
        II. Irreducibility of Cyclotomic Polynomials}
\author{F. Lemmermeyer}
\email{hb3@ix.urz.uni-heidelberg.de}
\address{M\"orikeweg 1, 73489 Jagstzell, Germany}

\begin{document}
\maketitle
\markboth{Harbingers of Artin's Reciprocity Law}
         {\today \hfil Franz Lemmermeyer}
\begin{center} \today \end{center}
\bigskip

In \cite{LH1}, we have presented the history of auxiliary
primes from Legendre's proof of the quadratic reciprocity 
law up to Artin's reciprocity law. We have also seen that
the proof of Artin's reciprocity law consists of several
steps, the first of which is the verification of the 
reciprocity law for cyclotomic extensions of $\Q$. In this
article we will show that this step can be identified with
one of Dedekind's proofs of the irreducibility of the
cyclotomic polynomial.

\section{Irreducibility of Cyclotomic Polynomials}

Let $\zeta$ be a primitive $m$-th root of unity, and let 
$\mu_m = \la \zeta \ra$ denote the group of $m$-th roots of unity. 
The polynomial $\Phi_m(X) = \prod_{(r,m) = 1} (X - \zeta^r)$ 
is called the $m$-th cyclotomic polynomial. Since every automorphism
of $\bQ/\Q$ maps a primitive $m$-th root of unity to another
primitive $m$-th root of unity, the coefficients of $\Phi_m(X)$ 
must be rational numbers, and actually are integers since $\zeta$
is an algebraic integer. 

The first proofs of the irreducibility of the cyclotomic polynomial
were obtained by Gauss, Eisenstein, Arndt, and Kronecker. Dedekind
gave three proofs: \cite{Ded1} was published in 1857, his second proof
was contained in his review of Bachmann's \cite{BKT}, and the last one
\cite{Ded2} in 1894. If Dedekind published a new proof of a classical
result then he did so because he thought that the new proof was 
superior to the known proofs. In fact, his third proof will turn
out to be a small first step in Artin's and Tate's proof of Artin's
reciprocity law. In this section we will go through Dedekind's 
third proof; we have updated the form, but not the content, of 
Dedekind's results.

Now fix a primitive $m$-th root of unity $\zeta$, and assume that
$\Phi_m(X)$ is reducible over $\Q$. Then $\zeta$ is a root of one 
of the factors of $\Phi_m(X)$, and if we let $f$ denote the minimal
polynomial of $\zeta$ we can write $\Phi_m(X) = f(X)g(X)$ for
polynomials $f, g \in \Q[X]$ with $f(\zeta) = 0$.

Let $L = \Q(\mu_m)$ denote the splitting field of $\Phi_m(X)$ and 
$K = \Q(\zeta)$ the splitting field of $f$; observe that $L$ contains 
all $m$-th roots of unity, and $K$ only those conjugate to $\zeta$. 
In particular, we have $K \subseteq L$, $(K:\Q) = \deg f$ and 
$(L:\Q) = \phi(m)$. 

Thus $\Phi_m(X)$ is irreducible if and only if $(K:\Q) \ge \phi(m)$.
The following lemma collects other ways of expressing this fact:

\begin{lem}\label{Lirr}
The following assertions are equivalent:
\begin{enumerate}
\item $\Phi_m(X)$ is irreducible over $\Q$;
\item $(K:\Q) \ge \phi(m)$ for $K = \Q(\zeta)$;
\item $K = L$, where $L$ is the splitting field of $\Phi_m(X)$;
\item for every $r$ coprime to $m$, the element $\zeta^r$ is a 
      conjugate of $\zeta$.
\item We have $\Gal(\Q(\zeta)/\Q) = \{\sigma_r: \gcd(r.m) = 1\}$,
      where $\sigma_r(\zeta) = \zeta^r$. 
\end{enumerate}
\end{lem}

In both of Dedekind's proofs of the irreducibility of $\Phi_m(X)$ it is
shown that (4) holds. The two main ingredients of Dedekind's third
proof are the following two lemmas:

\begin{lem}\label{Lin}
Let $F$ be a number field containing $\mu_m$, and let $\fp$ be
a prime ideal in $\cO_F$ coprime to $m$. Then the natural 
projection $\pi:\mu_m \lra (\cO_F/\fp)^\times$ is injective.
\end{lem}

The proof of Lemma \ref{Lin} will be given below; Dedekind's 
formulation of this result was the following:
\begin{quote}
Let $\fp$ be a prime ideal in a number field $\Omega$, let $p$
be the prime number divisible by $\fp$, and set $N \fp = p^f$,
where $f$ is the inertia degree of $\fp$. If $\alpha$ is a primitive
$m$-th root of unity in $\Omega$, and if we set $m = m'p'$, where
$p'$ is the highest power of $p$ dividing $m$, then $\alpha$
belongs to the exponent $m' \bmod \fp$.
\end{quote}
In the case we are interested in we have $p \nmid m$, hence $m' = m$;
the exponent to which a root of unity $\zeta$ belongs mod $\fp$
is the smallest positive integer $e$ such that 
$\zeta^e \equiv 1 \bmod \fp$. Thus Dedekind claims that the order
of $\zeta$ and that of $\zeta \bmod \fp$ coincide, which is exactly
the content of Lemma \ref{Lin} above.

The next lemma provides us with the existence of the Frobenius 
automorphism\footnote{In a letter to Dedekind, probably written 
on the first days of June 1882, Frobenius asked whether Dedekind 
knew a certain result concerning the Galois group of an extension 
of number fields $K/\Q$ and the decomposition of prime ideals in $K$. 
Dedekind replied on June 8 that he knew this result, and Frobenius 
gave his proof of the existence of the Frobenius substitution in his 
answer. Dedekind explained his own proof in his letter dated June 14, 
and this version was used by Frobenius in his publication \cite{Fro}.
For the relevant parts of these letters see \cite{DedF}.}, 
which is a well known result from the Galois theory of normal 
extensions of number fields:

\begin{lem}\label{LFr}
Let $L/K$ be a normal extension of number fields, $\fp$ a nonzero 
prime ideal in $\cO_K$ that does not ramify in $L/K$, and $\fP$ a
prime ideal in $\cO_L$ above $\fp$. The residue class fields
$\bK = \cO_K/\fp$ and $\bL = \cO_L/\fP$ are finite fields, and
the natural projection
$$ Z(\fP|\fp) \lra \Gal(\bL/\bK) $$
from the decomposition group $Z(\fP|\fp)$ to the Galois group of 
the (cyclic) extension $\bL/\bK$ of finite fields is an isomorphism. 
The unique preimage of the Frobenius automorphism of $\bK$ is
called the Frobenius automorphism of $\fp$. This Frobenius 
automorphism $\sigma$ is characterized by the property that
$\sigma(\alpha) \equiv \alpha^p \bmod \fp$ for every 
$\alpha \in \cO_K$.
\end{lem}

Now we claim

\begin{prop}\label{PD1}
The natural homomorphism $\phi:\Gal(K/\Q) \lra (\Z/m\Z)^\times$ 
sending $\sigma \in \Gal(K/\Q)$ to the residue class of 
$r \bmod m$ determined by $\sigma(\zeta) = \zeta^r$ is surjective.
\end{prop}

Since $\Gal(K/\Q)$ has order $(K:\Q)$ (recall that $K$ is the splitting
field of the factor $f$ of $\Phi_m$, hence normal), this implies
that $\Phi_m(X)$ is irreducible over $\Q$.

\begin{proof}[Proof of Prop. \ref{PD1}]
Since the coprime residue classes modulo $m$ are generated by 
classes represented by primes $p$ not dividing $m$, it is more than
enough to show that the residue classes generated by these $p$
are in the image (this reduction can already be found in 
Dedekind's first proof in \cite{Ded1}, as well as in van der 
Waerden's Algebra).

Let $p$ denote a prime number coprime to $m$. We have to 
show that there is a $\sigma \in \Gal(K/\Q)$ with 
$\sigma(\zeta) = \zeta^p$. Let $\sigma$ be the Frobenius automorphism 
for $p$; then $\sigma(\alpha) \equiv \alpha^p \bmod \fp$. Applying 
this to $\zeta$ we find $\sigma(\zeta) \equiv \zeta^p \bmod \fp$. 
Since $\sigma(\zeta)$ must be a root of $\Phi_m(X)$, we can write 
$\sigma(\zeta) = \zeta^s$ for some integer $s$ with $1 \le s < m$; 
then $\zeta^s \equiv \zeta^p \bmod \fp$, and Lemma \ref{Lin} implies 
that $s = p$, that is, $\sigma(\zeta) = \zeta^p$.
\end{proof}

Observe that we did not assume the existence of any primes lying in 
certain residue classes modulo $m$, but rather showed that if there
is a prime in some residue class, then it is the image of its
Frobenius automorphism. 

For proving Lemma \ref{Lin}, Dedekind uses the prime ideal 
factorization of the element $1-\zeta$ in $\Q(\mu_m)$:

\begin{lem}\label{LD1}
In the splitting field $\Q(\mu_m)$ of $\Phi_m(X)$, we have 
\begin{enumerate}
\item $\zeta - 1 = 0$ if $m = 1$.
\item $(\zeta-1)^{\phi(m)} = \eps p$ if $m$ is a power of $p$,
      where $\eps$ is a unit; moreover, if $\gcd(r,p) = \gcd(s,p) = 1$,
      then $\frac{\zeta^r-1}{\zeta^s-1}$ is a unit.  
\item $\zeta - 1 = \eps$ if $m$ is divisible by at least two
      distinct primes.
\end{enumerate} 
\end{lem}

These are simple properties of cyclotomic fields whose proofs
can be found in most textbooks on algebraic number theory.
In order to convince ourselves that Dedekind is not using
the irreducibility of $\Phi_m(X)$ along the way, let us derive
these results here:

\begin{proof}[Proof of Lemma \ref{LD1}]
Property (1) is trivial. For a proof of (2), let $r$ run through
all $\phi(m)$ coprime residue classes and set $rr' \equiv 1 \bmod m$.
The elements
$$ \frac{\zeta^r-1}{\zeta-1} \quad \text{and} \quad  
   \frac{\zeta-1}{\zeta^r-1} = \frac{\zeta^{rr'}-1}{\zeta^r-1} $$
are clearly integral, hence are units in the splitting field $L$ of 
$\Phi_m(X)$. Moreover, we have 
$$ \Phi_m(X) = \frac{X^m-1}{X^{m/p}-1} = \prod (X - \zeta^r), $$
where the product is over all $1 \le r < m = p^n$ with $\gcd(r,m) = 1$.
Writing 
$X^m-1 = (X^{\frac mp}-1)(X^{m-\frac mp} + \ldots + X^{\frac mp} + 1)$
and plugging in $X = \zeta$ then gives $p = \prod (1 - \zeta^r)$,
hence $(\zeta-1)^{\phi(m)} = p\eps$ for some unit $\eps \in \cO_L$.

Next, since both $\frac{\zeta^r-1}{\zeta-1}$ and 
$\frac{\zeta^s-1}{\zeta-1}$ are units whenever $r$ and $s$ are integers
coprime to $m$, so is their quotient $\frac{\zeta^r-1}{\zeta^s-1}$.

If finally $m = pqn$ is divisible by two distinct primes $p$ and $q$,
then $\zeta-1$ is a common divisor of $\zeta^{qn}-1 = \zeta_p -1$ and 
$\zeta^{pn}-1 = \zeta_q - 1$, hence a common divisor of $p = N(\zeta_p-1)$ 
and $q = N(\zeta_q-1)$, and thus a unit.
\end{proof}

Dedekind uses Lemma \ref{LD1} to give the following

\begin{proof}[Proof of Lemma \ref{Lin}]
Let $f \ge 1$ be the minimal integer with $\zeta^f \in \ker \pi$,
that is, with $\zeta^f-1 \in \fp$. Then $\zeta^f$ is a primitive
$n$-th root of unity for some $n \mid m$, and there are several cases:
\begin{enumerate}
\item $n$ is divisible by two distinct primes; then the prime ideal
      $\fp$ contains the unit $\zeta^f-1$: contradiction. 
\item $n$ is a power of a prime $p$; then $\zeta^f-1 \in \fp$, where
      $\fp$ is a prime ideal above $p \mid m$ contradicting our 
      assumption.  
\item $n = 1$: then $f = m$, hence $\ker \pi = 1$, and this is 
      what we wanted to prove. 
\end{enumerate}
\end{proof}

\subsection*{An Irreducibility Proof based on Lemma \ref{LD1}}
The irreducibility of the cyclotomic polynomial $\Phi_m(X)$ for
prime powers $m$ follows from part (2) of Lemma \ref{LD1}:
the equation $p = \eps^{-1}(\zeta-1)^{\phi(m)}$ is valid in $L$, 
but since $p$ and $\zeta$ are in $K$, so is $\eps$; thus $(p)$ 
is the $\phi(m)$-th power of the ideal $(\zeta-1)$ in $K$, and we 
must have $(K:\Q) \ge \phi(m)$.

This can easily be extended to a proof of the irreducibility of
$\Phi_m(X)$ in the general case: we have to show that 
$(\Q(\zeta_m):\Q) \ge \phi(m)$. The idea of the proof becomes
clear enough by treating the case where $m = PQ$ is a product 
of two prime powers $P$ and $Q$.  Then $\Q(\zeta_m)$ is the 
compositum of the fields $\Q(\zeta_P)$ and  $\Q(\zeta_Q)$,
which have degrees $\phi(P)$ and $\phi(Q)$, respectively. Since
the fields $\Q(\zeta_P)$ and $\Q(\zeta_Q)$ are independent
($p$ is fully ramified in $\Q(\zeta_P)$ and unramified in 
$\Q(\zeta_Q)$), we must have 
$$ (\Q(\zeta_m):\Q) = (\Q(\zeta_P):\Q)(\Q(\zeta_Q):\Q)
      = \phi(P) \phi(Q) = \phi(m). $$

The fact that Dedekind chose not to give this proof does not mean
that he did not see it: Dedekind avoided artificial tricks and
shortcuts whenever he thought that this would make the proof
less conceptional. In the case of the irreducibility of $\Phi_m(X)$
he also wanted to avoid arguments that do not carry over to the
theory of complex multiplication: there one had to prove the
irreducibility of certain polynomials (coming from the theory
of elliptic curves and modular functions) whose roots generate
abelian extensions of complex quadratic number fields.
 
\section{The Proof of Artin's Reciprocity Law.}

The proof of Artin's reciprocity law is quite involved; the first 
(essential) step is verifying the reciprocity law for cyclotomic 
extensions of $\Q$; this step (see e.g. Artin \& Tate 
\cite[Lemma 1, p. 42]{AT}) is nothing but Dedekind's proof of the
irreducibility of the cyclotomic polynomial discussed above: 

\begin{lem}
Let $\zeta$ be a primitive $m$-th root of unity. An automorphism
of $\Q(\zeta)/\Q$ sends $\zeta$ into a power $\zeta^n$ where $n$
is coprime to $m$. Conversely, to any given $n$ prime to $m$ there
is an automorphism $\sigma$ such that $\zeta^\sigma = \zeta^n$.
\end{lem}

\begin{proof}
The first part of the lemma is trivial. As for the second part,
it suffices to prove the statement if $n$ is a prime $p$ that 
does not divide $m$. Recall that $\zeta$ satisfies the equation 
$x^m - 1 = f(x) = 0$, and $f'(\zeta) = m\zeta^{m-1}$ is prime 
to $p$. The local field $\Q_p(\zeta)/\Q_p$ is therefore 
unramified. The Frobenius substitution sends $\zeta$ into an
$m$-th root of unity that is congruent to $\zeta^p$. Since 
$f'(\zeta) = \prod_\mu (\zeta- \zeta^\mu)$ is prime to $p$ it 
follows that no two $m$-th roots of unity are in the same
residue class; $\zeta^p$ is therefore the image of $\zeta$
under the Frobenius substitution. This automorphism of the 
local field is induced by an automorphism of the global field
$\Q(\zeta)/\Q$ and this proves the lemma.
\end{proof}

In this proof, Dedekind's Lemma \ref{Lin} is condensed into the 
observation that, for $f(x) = x^m-1$, the element 
$f'(\zeta) = \prod_{r \ne 1} (\zeta - \zeta^r)$ is coprime to $p$.

\subsection*{Emmy Noether's Comments}
The fact that Dedekind's proof of the irreducibility of the 
cyclotomic equation is connected with Artin's reciprocity law
was first noticed by Emmy Noether. In her comments on Dedekind's
proof in his Collected Works she writes:
\begin{quote}
The proof of the irreducibility of the cyclotomic equation given here
is based on the following two facts:
\begin{enumerate}
\item A primitive $m$-th root of unity remains primitive modulo
      each prime ideal $\fp$ not dividing $m$ in the field of 
      $m$-th roots of unity.
\item There is a substitution $\psi_0$ of the decomposition
      group of $\fp$ for which $\psi_0(\omega) \equiv \omega^p \bmod \fp$
      for every integral $\omega$ in $K$.
\end{enumerate} 
The map given in (2) from the Galois group and the class group -- from 
which irreducibility follows immediately -- is the map given by Artin's
reciprocity law, but in a very weak and therefore elementary form.
\end{quote}

In fact, Prop. \ref{PD1} gives an isomorphism
$\Gal(K/\Q) \simeq (\Z/m\Z)^\times$. Composed with the isomorphism
$(\Z/m\Z)^\times \simeq \Cl_\Q\{m\infty\}$ this gives a canonical 
isomorphism between $\Cl_\Q\{m\infty\}$ and $\Gal(K/\Q)$.

Finally let us remark that the ``small piece of Artin's Reciprocity
Law for cyclotomic extensions'', which, according to \cite[p. 54]{FJ}, 
is contained in Chebotarev's article \cite[p. 200]{Cheb}, is based on 
the irreducibility of the cyclotomic equation, which Tschebotareff 
takes for granted.

\bigskip

\end{document}